\documentclass[10pt]{article}
\usepackage{hyperref}
\usepackage{amsmath}
\usepackage[dvips]{graphicx}
\usepackage{amsthm}
\usepackage{epsfig}
\usepackage{amssymb}
\usepackage{dsfont}

\newtheorem{theorem}{Theorem}[section]

\begin{document}

\begin{center}
{\Large  Some existence and nonexistence results for a Schr\"odinger-Poisson type system}
\end{center}

\vskip 5mm

\begin{center}
{\sc Yutian Lei} \\
\vskip 3mm
Institute of Mathematics\\
School of Mathematical Sciences\\
Nanjing Normal University\\
Nanjing, 210023, China\\
Email:leiyutian@njnu.edu.cn
\end{center}

\vskip 5mm {\leftskip5mm\rightskip5mm \normalsize
\noindent{\bf{Abstract}}
In this paper, we study the Schr\"odinger-Poisson system
$$
\left \{
   \begin{array}{l}
      -\Delta u=\sqrt{p}u^{p-1}v, \quad u>0 \quad in \quad R^n,\\
      -\Delta v=\sqrt{p}u^p, \quad v>0 \quad in \quad R^n
   \end{array}
   \right.
$$
with $n \geq 3$ and $p>1$. We investigate the existence and the
nonexistence of positive classical solutions with the help of an integral system involving
the Newton potential
$$
\left \{
   \begin{array}{l}
      u(x)=c_1\displaystyle\int_{R^n}\frac{u^{p-1}(y)v(y)dy}{|x-y|^{n-2}}, \quad u>0 \quad in \quad R^n,\\
      v(x)=c_2\displaystyle\int_{R^n}\frac{u^p(y)dy}{|x-y|^{n-2}} \quad v>0 \quad in \quad R^n.
   \end{array}
   \right.
$$
First, the system has no solution when $p\leq \frac{n}{n-2}$. When $p>\frac{n}{n-2}$,
the system has a singular solution on $R^n \setminus \{0\}$ with
slow asymptotic rate $\frac{2}{p-1}$. When $p<\frac{n+2}{n-2}$,
the system has no solution in $L^{\frac{n(p-1)}{2}}(R^n)$. In fact, if the system has solutions
in $L^{\frac{n(p-1)}{2}}(R^n)$, then $p=\frac{n+2}{n-2}$ and all the positive classical
solutions can be classified as $u(x)=v(x)=c(\frac{t}{t^2+|x-x^*|^2})^{\frac{n-2}{2}}$,
where $c,t$ are positive constants. When $p>\frac{n+2}{n-2}$, by the shooting method
and the Pohozaev identity, we find another pair of radial solution
$(u,v)$ satisfying $u \equiv v$ and decaying with slow rate $\frac{2}{p-1}$.

\par
\noindent{\bf{Keywords}}: Schr\"odinger-Poisson equation, positive classical
solution, existence and nonexistence, critical and noncritical conditions
\par
{\bf{MSC2010}} 35J10, 35J47, 35Q55, 45E10, 45G05}

\newtheorem{proposition}[theorem]{Proposition}

\renewcommand{\theequation}{\thesection.\arabic{equation}}
\catcode`@=11
\@addtoreset{equation}{section}
\catcode`@=12

\section{Introduction}  

Recently, many authors devoted to study the nonlocal stationary Schr\"odinger equation
\begin{equation} \label{PDE}
-\Delta u=pu^{p-1}(|x|^{2-n}*u^p), \quad u>0 \quad in \quad R^n,
\end{equation}
where $n \geq 3$ and $p>1$. In particular, Moroz and Van Schaftingen \cite{MVS}
studied the existence of the supersolutions and listed several sufficient conditions.

Equation (\ref{PDE}) is defined in many places. One is the example
3.2.8 in the book \cite{Caz}. A more general form is the Choquard type equation
in the papers \cite{Lei} and \cite{MZ}. Another interesting work related to (\ref{PDE}) is the
paper \cite{JL-Z} and the references therein.
Equation (\ref{PDE}) is also helpful in understanding the blowing
up or the global existence and scattering of the solutions of the dynamic Hartree
equation (cf. \cite{LMZ}), which arises in the study of boson stars and other physical phenomena,
and also appears as a continuous-limit model for mesoscopic molecular structures
in chemistry. Such an equation also arises in the Hartree-Fock theory of the
nonlinear Schr\"odinger equations (cf. \cite{LS}). More related mathematical and
physical background can be found in \cite{BAM}, \cite{GV}, \cite{Na}, \cite{Smets}
and the references therein.

Since (\ref{PDE}) has a convolution term,
it seems difficult to investigate the existence directly.
Write
$$
v(x)=\sqrt{p}\int_{R^n}\frac{u^p(y)dy}{|x-y|^{n-2}}.
$$
Then $v>0$ in $R^n$. Noting the relation between the Newton potential and the
convolution properties of Dirac function, we see that
$$
-\Delta v(x)=\sqrt{p}(-\Delta |x|^{2-n})*u^p =\sqrt{p}\delta_x
*u^p=\sqrt{p}u^p(x),
$$
where $\delta_x$ is the Dirac mass at $x$.
Thus, the positive solution of (\ref{PDE}) must satisfy the
following system
\begin{equation} \label{Sys}
\left \{
   \begin{array}{l}
      -\Delta u=\sqrt{p}u^{p-1}v, \quad u>0 \quad in \quad R^n,\\
      -\Delta v=\sqrt{p}u^p, \quad v>0 \quad in \quad R^n.
   \end{array}
   \right.
\end{equation}
It is a simplified model of the Schr\"odinger-Piosson system (cf.
\cite{Amst}, \cite{HW} and references therein).

Quittner and Souplet \cite{QS} studied
positive solutions of another PDE system
\begin{equation} \label{fgf}
\left \{
   \begin{array}{l}
      -\Delta u=v^p u^r, \quad u>0 \quad in \quad R^n,\\
      -\Delta v=v^s u^q, \quad v>0 \quad in \quad R^n.
   \end{array}
   \right.
\end{equation}
They proved the following results:

(R1) If $n \geq 3$, $p-s=q-r \geq 0$ and $0 \leq r,s \leq \frac{n}{n-2}$,
then positive solutions $u,v$ satisfy $u \equiv v$.

(R2) If $n \geq 3$, $p-s=q-r \geq 0$, then nonnegative solutions $u,v$ satisfy
$u \geq v$ or $v \geq u$.

\vskip 5mm
In this paper, we study the
existence of positive classical solutions to the
Schr\"odinger-Piosson type system (\ref{Sys}). Besides (R1) and (R2),
we have further existence and nonexistence results and state them in three cases.

\vskip 5mm

First we consider the subcritical case $p \in (1,2^*-1)$, where $2^*
=\frac{2n}{n-2}$.
We have the following nonexistence results.

\begin{theorem} \label{th1.1}
(1)If $p \leq \frac{n}{n-2}$, then (\ref{Sys}) has no positive solution.

(2) If $2 \leq p<2^*-1$, then (\ref{Sys}) has no positive solution in
$L^{\frac{n(p-1)}{2}}(R^n)$.
\end{theorem}

\paragraph{Remark.}
\begin{enumerate}
\item When $u \equiv v$, (\ref{Sys}) is reduced to the Lane-Emden type equation
\begin{equation} \label{LE}
-\Delta u=\sqrt{p}u^p, \quad u>0~in~R^n.
\end{equation}
Its nonexistence results similar to Theorem \ref{th1.1} can be found in \cite{GS}.
For the system (\ref{Sys}), we here introduce an integral system to
investigate the nonexistence. By the same argument in \cite{CLO} and \cite{SIAM}, we know
that the solution of (\ref{Sys}) satisfies the following integral
system involving Newton potentials
\begin{equation} \label{IE}
\left \{
   \begin{array}{l}
      u(x)=c_1\displaystyle\int_{R^n}\frac{u^{p-1}(y)v(y)dy}{|x-y|^{n-2}}, \quad u>0 \quad in \quad R^n,\\
      v(x)=c_2\displaystyle\int_{R^n}\frac{u^p(y)dy}{|x-y|^{n-2}} \quad v>0 \quad in \quad R^n,
   \end{array}
   \right.
\end{equation}
for some constants $c_1,c_2>0$. By this integral system we can verify the nonexistence of positive solutions
when $p\leq \frac{n}{n-2}$. When $p>\frac{n}{n-2}$, then (\ref{Sys}) always has singular solutions
on $R^n \setminus \{0\}$ (cf. Theorem \ref{th4.1}).

\item The integral system (\ref{IE}) is invariant after
translation. Its $L^{\frac{n(p-1)}{2}}(R^n)$-solutions are radially
symmetric about some point $x^* \in R^n$. In the noncritical case,
Kelvin transformation breaks the translation invariant of
(\ref{IE}). On the other hand, the
$L^{\frac{n(p-1)}{2}}(R^n)$-solutions are still radially symmetric
about the origin in the subcritical case. These facts lead to a
contradiction and hence $L^{\frac{n(p-1)}{2}}(R^n)$-solution does
not exist.
\end{enumerate}

\vskip 5mm Next, we consider the critical case and classify the
$L^{\frac{n(p-1)}{2}}(R^n)$-solutions.

\begin{theorem} \label{th1.2}
Let $u$ be a classical solution of (\ref{Sys}). Then the following items are
equivalent to each other

(1) $u \in L^{\frac{n(p-1)}{2}}(R^n)$;

(2) $u$ is bounded and decaying with the fast rate $n-2$;

(3) $u$ belongs to the homogeneous Sobolev space
$\mathcal{D}^{1,2}(R^n)$;

(4) $u \in L^{2^*}(R^n)$ and $p=2^*-1$;

(5) $u(x) = v(x)
=c(\frac{t}{t^2+|x-x^*|^2})^{\frac{n-2}{2}}$ with constants
$c,t>0$.
\end{theorem}

In the special case $u \equiv v$, we recall the Lane-Emden equation (\ref{LE}).
The classification of the solutions of this single equation
has provided an important ingredient in the study of the conformal geometry,
such as the prescribing scalar curvature problem and the extremal functions of
the Sobolev inequalities. It was studied rather extensively (cf. \cite{CGS},
\cite{GNN}, \cite{GS}, \cite{LN} and the reference therein). In
particular, Chen and Li \cite{CL} proved that all the positive solutions of
(\ref{LE}) with the critical exponent $p=2^*-1$ must be the form as
$$
u(x)
=c(\frac{t}{t^2+|x-x^*|^2})^{\frac{n-2}{2}}
$$
with constants $c,t>0$. For the system (\ref{Sys}), we expect to prove both $u \equiv v$
and $p=2^*-1$. Once these results are verified, we can use the result
in \cite{CL} to classify the positive solutions of (\ref{Sys}).

\vskip 5mm
At last, we consider the supercritical case $p>2^*-1$.
We use the shooting method and Pohozaev identity to prove the following result.

\begin{theorem} \label{th1.3}
When $p>2^*-1$ and $p \geq 2$, we can find radial solutions $u,v$
decaying with the slow rate $\frac{2}{p-1}$ when $|x| \to \infty$.
\end{theorem}

Not all the solutions in the supercritical case are radially symmetric.
In section 4 we introduce an example to show that some bounded solutions
are neither radial nor decaying when $|x| \to \infty$.

\vskip 5mm
Another integral system as (\ref{IE}) is the following Lane-Emden type equations
\begin{equation} \label{other}
\left \{
   \begin{array}{l}
      u(x)=c_1\displaystyle\int_{R^n}\frac{v^q(y)dy}{|x-y|^{\lambda}}, \quad u>0 \quad in \quad R^n,\\
      v(x)=c_2\displaystyle\int_{R^n}\frac{u^p(y)dy}{|x-y|^{\lambda}}, \quad v>0 \quad in \quad R^n.
   \end{array}
   \right.
\end{equation}
It is essential in studying the extremal
functions of the Hardy-Littlewood-Sobolev inequality (cf. \cite{Lieb})
$$
    \int_{R^n}\int_{R^n} \displaystyle{\frac{f(x)g(y)}{|x-y|^{\lambda}}} dx dy
    \leq
    C(n,s,\lambda) ||f||_r||g||_s
$$
with
$0 <\lambda < n$, $1< s, r  < \infty$, $f \in L^r(R^n)$ and $g \in
L^s(R^n)$, $\frac{1}{r} +\frac{1}{s} + \frac{\lambda}{n}= 2$.

Define $Tg(x)=\displaystyle{\int}_{R^n} \displaystyle{\frac{g(y)}{|x-y|^{n-\alpha}} }dy$
with $\alpha=n-\lambda$. The Hardy-Littlewood-Sobolev inequality becomes
\[
 ||Tg||_p \leq C(n,s,\alpha) ||g||_{\frac{np}{n+\alpha p}},
\]
where $\frac{n}{n-\alpha}<p< \infty$, and $1<s< n/{\alpha}$. This inequality
will be used in this paper to research the radial symmetry and the
integrability of the solutions of (\ref{IE}).

In the critical case, the classification results for the single equation of (\ref{other})
can be found in \cite{CLO} and \cite{YLi}. In particular,
the method of moving planes of integral forms was introduced in \cite{CLO}.
It has become a powerful tool to handle the qualitative properties including the
existence and the nonexistence, the radial symmetry, and the priori estimates.
Jin and Li \cite{JL} applied a regularity lifting lemma by the contraction maps
to obtain the optimal integrability of positive regular solutions.
Based on this result, \cite{LLM} estimated the fast decay rates when $|x| \to
\infty$.

\section{Subcritical case}

In this section, we always assume that $p<2^*-1$.

\begin{theorem} \label{th2.1} (Liouville theorem)
If $p\leq \frac{n}{n-2}$, then there does
not exist any positive solution of (\ref{IE}).
\end{theorem}

\begin{proof}
The idea in \cite{SIAM} can be used here.

If $u,v$ are positive solutions, we can deduce a contradiction.

(1) When $p<\frac{n}{n-2}$, it follows $n-2-\frac{2}{p-1}<0$.

Set $a_0=n-2$, $a_{j+1}=(2p-1)a_j-4$, $b_j=pa_j-2$ for $j=0,1,2,\cdots$.

By (\ref{IE}), we have
\begin{equation} \label{aq}
u(x) \geq c\int_{B_1(0)}\frac{u^{p-1}(y)v(y)dy}{|x-y|^{n-2}}
\geq \frac{c}{|x|^{a_0}},
\end{equation}
$$
v(x) \geq c\int_{B(x,|x|/2)}\frac{u^p(y)dy}{|x-y|^{n-2}}
\geq \frac{c}{|x|^{b_0}},
$$
$$
u(x) \geq c\int_{B(x,|x|/2)}\frac{u^{p-1}(y)v(y)dy}{|x-y|^{n-2}}
\geq \frac{c}{|x|^{a_1}},
$$
$$
\cdots \quad \cdots
$$
By induction, we have $u(x) \geq \frac{c}{|x|^{a_j}}$ as long as
$a_j>0$. Noting $p>1$ and $n-2-\frac{2}{p-1}<0$, from
$$\begin{array}{ll}
a_j&=a_0(2p-1)^j-4[(2p-1)^{j-1}+(2p-1)^{j-2}+\cdots+(2p-1)^0]\\[3mm]
&=(n-2-\frac{2}{p-1})(2p-1)^j+\frac{2}{p-1},
\end{array}
$$
we can find some $j_0>0$ such that $a_{j_0} \leq 0$. This leads to
$u(x)=\infty$ and yields a contradiction.

(2) When $p=\frac{n}{n-2}$, from
$$
u(x) \geq c\int_{B_R(0)}\frac{u^{p-1}(y)v(y)dy}{|x-y|^{n-2}}
\geq \frac{c}{(R+|x|)^{n-2}}\int_{B_R(0)}u^{p-1}(y)v(y)dy,
$$
we deduce
\begin{equation} \label{iin}
\begin{array}{ll}
\displaystyle\int_{B_R(0)}u^p(x)dx
&\geq \displaystyle\int_{B_R(0)}\frac{cdx}{(R+|x|)^n}
(\int_{B_R(0)}u^{p-1}(y)v(y)dy)^p\\[3mm]
&\geq c(\displaystyle\int_{B_R(0)}u^{p-1}(y)v(y)dy)^p.
\end{array}
\end{equation}
Here $c$ is independent of $R$. Similarly, from
\begin{equation} \label{quaf}
v(x) \geq \frac{c}{(R+|x|)^{n-2}}\int_{B_R(0)}u^p(y)dy,
\end{equation}
we also deduce
$$
\int_{B_R(0)}u^{p-1}(x)v(x)dx
\geq \int_{B_R(0)}\frac{cu^{p-1}(x)dx}{(R+|x|)^{n-2}}\int_{B_R(0)}u^p(y)dy.
$$
Using (\ref{aq}), (\ref{iin}) and noting $p=\frac{n}{n-2}$, we get
$$
\int_{B_R(0)}u^{p-1}(x)v(x)dx
\geq c(\int_{B_R(0)}u^{p-1}(y)v(y)dy)^p,
$$
which implies $u^{p-1}v \in L^1(R^n)$ if we let $R \to \infty$.

Multiplying (\ref{quaf}) by $u^{p-1}$ and integrating on $A_R:=B_R(0)
\setminus B_{R/2}(0)$, we still have
$$
\int_{A_R}u^{p-1}(x)v(x)dx
\geq c(\int_{B_R(0)}u^{p-1}(y)v(y)dy)^p.
$$
Letting $R \to \infty$ and noting $u^{p-1}v \in L^1(R^n)$,
we obtain $\|u^{p-1}v\|_{L^1(R^n)}=0$. It is impossible.
\end{proof}

By using the method of moving planes in integral forms,
which was established by Chen-Li-Ou (cf. \cite{CLO} and \cite{CLO4}),
we prove a radial symmetry result.

\begin{theorem} \label{prop2.1}
Let $p \geq 2$. If $h \geq 0$ satisfies
\begin{equation} \label{hhh}
\int_{R^n}|y|^{-nh/2}u^{n(p-1)/2}(y)dy<\infty, \quad
\int_{R^n}|y|^{-nh/2}v^{n(p-1)/2}(y)dy<\infty.
\end{equation}
Then the positive continuous solutions of
\begin{equation} \label{wIE}
\left \{
   \begin{array}{l}
      u(x)=c_1\displaystyle\int_{R^n}\frac{u^{p-1}(y)v(y)dy}{|y|^h|x-y|^{n-2}}, \quad u>0 \quad in \quad R^n,\\
      v(x)=c_2\displaystyle\int_{R^n}\frac{u^p(y)dy}{|y|^h|x-y|^{n-2}} \quad v>0 \quad in \quad R^n,
   \end{array}
   \right.
\end{equation}
are radially symmetric and decreasing about $x^* \in R^n$. Moreover, $x^*=0$ as long as $h>0$.
\end{theorem}

\begin{proof}
For some real number $\lambda$, define
$\Sigma_\lambda:=\{x=(x_1,\ldots,x_n);x_1>\lambda\}$,
$x^\lambda=(2\lambda-x_1,x_2,\ldots,x_n)$,
$u_\lambda(x)=u(x^\lambda)$, $\Sigma_\lambda^u=\{x \in \Sigma_\lambda|u(x) \leq u_\lambda(x)\}$,
$\Sigma_\lambda^v=\{x \in \Sigma_\lambda|v(x) \leq
v_\lambda(x)\}$.
It is not difficult to see that
\begin{equation} \label{ggg}
\begin{array}{ll}
&\quad u_\lambda(x)-u(x)\\[3mm]
&= c_1\displaystyle\int_{\Sigma_\lambda} \left
      (\frac{1}{|x-y|^{n-2}}-\frac{1}{|x^{\lambda}-y|^{n-2}}\right)
      \frac{1}{|y^{\lambda}|^h}(v_{\lambda}u_{\lambda}^{p-1}-vu^{p-1})dy\\[3mm]
 & -  c_1\displaystyle\int_{\Sigma_\lambda}
      \left(\frac{1}{|x-y|^{n-2}}-\frac{1}{|x^\lambda-y|^{n-2}}\right)
      \left(\frac{1}{|y|^h}-\frac{1}{|y^\lambda|^h}\right)vu^{p-1}dy.
\end{array}
\end{equation}
Since the second term of the right hand side is nonpositive, from
the definition of $\Sigma_\lambda^{u}$ and $\Sigma_\lambda^{v}$, it follows
$$\begin{array}{ll}
&\quad  u_\lambda(x)-u(x)\\[3mm]
      &\leq c\displaystyle\int_{\Sigma_\lambda} \left
      (\frac{1}{|x-y|^{n-2}}-\frac{1}{|x^{\lambda}-y|^{n-2}}\right)
      \frac{1}{|y^{\lambda}|^h}(v_{\lambda}u_{\lambda}^{p-1}-vu^{p-1})dy\\[3mm]
      &\leq c\displaystyle\int_{\Sigma_\lambda^{v}}
    \frac{1}{|x-y|^{n-2}}\frac{1}{|y^\lambda|^h}u_{\lambda}^{p-1}(v_\lambda-v)(y)dy\\[3mm]
    &\quad+c\displaystyle\int_{\Sigma_\lambda^{u}} \frac{1}{|x-y|^{n-2}}\frac{1}{|y^\lambda|^h} v(u_\lambda^{p-1}-
     u^{p-1})(y)dy.
\end{array}
$$
Using the Hardy-Littlewood-Sobolev inequality and the H\"older inequality, we have
$$\begin{array}{ll}
&\quad \|u_\lambda-u\|_{L^s(\Sigma _\lambda^{u})} \\[3mm]
&\leq C\||y|^{-h}u_\lambda^{p-1}(v_\lambda-v)\|_{L^{\frac{ns}{n+2s}}(\Sigma _\lambda^{v})}
+C\||y|^{-h}vu_\lambda^{p-2}(u_\lambda-u)\|_{L^{\frac{ns}{n+2s}}(\Sigma _\lambda^{u})}\\[3mm]
& \leq C\||y|^{-h}u_\lambda^{p-1}\|_{L^{\frac{n}{2}}(\Sigma _\lambda^{v})}
\|(v_\lambda-v)\|_{L^s(\Sigma _\lambda^{v})}\\[3mm]
&\quad +C\||y|^{-h}u_\lambda^{p-2}v\|_{L^{\frac{n}{2}}(\Sigma _\lambda^{u})}
\|(u_\lambda-u)\|_{L^s(\Sigma _\lambda^{u})}.
\end{array}
$$
Similarly, we also obtain
$$
\|v_\lambda-v\|_{L^s(\Sigma _\lambda^{v})}
\leq C\||y|^{-h}u_\lambda^{p-1}\|_{L^{\frac{n}{2}}(\Sigma _\lambda^{u})}
\|(u_\lambda-u)\|_{L^s(\Sigma _\lambda^{u})}.
$$
By (\ref{hhh}), as $\lambda \to -\infty$,
$$
C\||y|^{-h}u_\lambda^{p-1}\|_{L^{\frac{n}{2}}(\Sigma _\lambda^{v})} \leq \frac{1}{4},
\quad \||y|^{-h}u_\lambda^{p-2}v\|_{L^{\frac{n}{2}}(\Sigma _\lambda^{u})}\leq\frac{1}{4}.
$$
Combining these results, we can see that $\Sigma _\lambda^{u}$ and $\Sigma _\lambda^{v}$
are empty set as long as $\lambda$ is near $-\infty$.

Suppose that at $\lambda_0<0$, we have $u(x) \geq u_{\lambda_0}(x)$ and $v(x) \geq v_{\lambda_0}(x)$
but $u(x) \not\equiv u_{\lambda_0}(x)$ and $v(x) \not\equiv v_{\lambda_0}(x)$ on $\Sigma_{\lambda_0}$.
By the same argument above, we can prove that there exists an $\epsilon>0$,
such that $u(x) \geq  u_\lambda(x)$ and $v(x) \geq v_{\lambda}(x)$ on $\Sigma_\lambda$ for
all $\lambda \in [\lambda_0,\lambda_0+\epsilon)$. Therefore, we can move plane $x_1=\lambda$ to the right as long as
$$
u(x) \geq u_\lambda(x) \quad \textrm{and} \quad v(x) \geq v_\lambda(x)
$$
hold on $\Sigma_\lambda$.
If the plane stops at $x_1=\lambda_0$ for some $\lambda_0 < 0$, then $u(x)$ and $v(x)$ must be
radially symmetric and decreasing about the plane $x_1=\lambda_0$.
Otherwise, we can move the plane all the way to $x_1=0$.
Since the direction of $x_1$ can be chosen arbitrarily, we obtain that
$u(x),v(x)$ are radially symmetric and decreasing about some $x^* \in R^n$.

If $h\neq 0$, we claim $x^*=0$. Otherwise, we can find
$\lambda_0<0$ such that $x_1=\lambda_0$ is the stopped plane.
From (\ref{ggg}), we get
$$\begin{array}{ll}
0&=u_{\lambda_0}(x)-u(x)\\[3mm]
&= -  c_1\displaystyle\int_{\Sigma_{\lambda_0}}
      \left(\frac{1}{|x-y|^{n-2}}-\frac{1}{|x^{\lambda_0}-y|^{n-2}}\right)
      \left(\frac{1}{|y|^h}-\frac{1}{|y^{\lambda_0}|^h}\right)vu^{p-1}dy \\[3mm]
&\neq 0.
      \end{array}
$$
It is impossible.
\end{proof}

\begin{theorem} \label{th2.2}
If $2 \leq p <2^*-1$, then there does not exist any positive solution of (\ref{IE})
in $L^{n(p-1)/2}(R^n)$.
\end{theorem}

\begin{proof}
{\it Step 1.} Suppose $u,v$ are the $L^{\frac{n(p-1)}{2}}(R^n)$-solutions of (\ref{IE}).
According to Theorem \ref{prop2.1} with $h=0$, we see that they are radially symmetric
about $x^* \in R^n$. Since (\ref{IE}) is invariant after translation, $x^*$ can be chosen
arbitrarily.

{\it Step 2.}
Take the Kelvin transformation of $u,v$
$$
\tilde{u}(x)=\frac{1}{|x|^{n-2}}u(\frac{x}{|x|^2}),
\quad
\tilde{u}(x)=\frac{1}{|x|^{n-2}}u(\frac{x}{|x|^2}).
$$
By (\ref{IE}), we see that $\tilde{u},\tilde{v}$ solve (\ref{wIE}) with
$h=n+2-p(n-2)$. In view of $p<2^*-1$, it follows $h>0$. In addition, from
$u,v \in L^{\frac{n(p-1)}{2}}(R^n)$, we see that (\ref{hhh}) for $\tilde{u},\tilde{v}$ is true.
According to Theorem \ref{prop2.1}, $\tilde{u},\tilde{v}$ are also radially symmetric but
the center point $x^*$ must be the origin. So the translation invariant is
absent. By the same argument of Theorem 3 in \cite{CLO4}, we can also deduce
a contradiction.
\end{proof}

\section{Critical case}

Theorem \ref{th2.1} implies that
\begin{equation} \label{base}
p>\frac{n}{n-2}
\end{equation}
is the necessary condition of the existence of positive solutions
for (\ref{Sys}). Therefore, we can easily see that
$$
u(x)=v(x)=c(\frac{t}{t^2+|x-x^*|^2})^{\frac{n-2}{2}}
$$
is a positive solution of (\ref{Sys}) in $L^{n(p-1)/2}(R^n)$. However,
it is nontrivial to show that all solutions of (\ref{Sys}) in $L^{n(p-1)/2}(R^n)$
are the form above. In this section, we prove this conclusion.

\begin{theorem} \label{th3.1}
If $u \in L^{\frac{n(p-1)}{2}}(R^n)$ solves (\ref{IE}), then

(R1) $u,v \in L^s(R^n)$ for all $s>\frac{n}{n-2}$;
$u,v \not\in L^s(R^n)$ for all $s\leq \frac{n}{n-2}$;

(R2) $u,v$ are bounded and
\begin{equation} \label{boundary}
\lim_{|x| \to \infty}u(x)=\lim_{|x| \to \infty}v(x)=0;
\end{equation}

(R3) $u,v$ are radially symmetric and decreasing about some point $x^* \in R^n$.
\end{theorem}

\begin{proof}

(1) By the Hardy-Littlewood-Sobolev inequality, from $u \in L^{\frac{n(p-1)}{2}}(R^n)$
we can deduce $v \in L^{\frac{n(p-1)}{2}}(R^n)$.

Write $w=u+v$. Then $w \in L^{\frac{n(p-1)}{2}}(R^n)$.
From (\ref{IE}), it follows that $w$ satisfies
\begin{equation} \label{W}
w(x)=K(x)\int_{R^n}\frac{w^p(y)dy}{|x-y|^{n-2}}.
\end{equation}
Here $K(x)>0$ is upper bounded. Set $w_A=w$ as $|x|>A$ or $w>A$;
$w=0$ as $|x| \leq A$ and $w \leq A$.
For $f \in L^s(R^n)$ with $s>\frac{n}{n-2}$, define
$$\begin{array}{ll}
&Tf(x):=K(x)\displaystyle\int_{R^n}\frac{w_A^{p-1}(y)f(y)dy}{|x-y|^{n-2}},\\[3mm]
&F(x):=K(x)\displaystyle\int_{R^n}\frac{(w-w_A)^p(y)dy}{|x-y|^{n-2}}.
\end{array}
$$
Therefore, $w$ solves the operator equation
$$
f=Tf+F.
$$
By the Hardy-Littlewood-Sobolev inequality, we get
$$
\|Tf\|_s \leq C\|w_A^{p-1}f\|_{\frac{ns}{n+2s}}
\leq C\|w_A\|_{\frac{n(p-1)}{2}}^{p-1}\|f\|_s,
$$
and
$$
\|F\|_s \leq C\|w-w_A\|_{\frac{nps}{n+2s}}^p<\infty.
$$
Take $A$ suitably large such that
$C\|w_A\|_{\frac{n(p-1)}{2}}^{p-1}<1$. Thus, $T$ is a contraction map
from $L^s(R^n)$ to itself. In view of $\frac{n(p-1)}{2}>\frac{n}{n-2}$
(which is implied by (\ref{base})), $T$ is also a contraction map
from $L^{n(p-1)/2}(R^n)$ to itself. By the lifting lemma on the regularity
(cf. Lemma 2.1 in \cite{JL}), we obtain $w \in L^s(R^n)$ for $s>\frac{n}{n-2}$.
Thus, $u,v \in L^s(R^n)$ for $s>\frac{n}{n-2}$.

On the other hand, if $s \leq \frac{n}{n-2}$, by (\ref{aq}) we have
$$
\|u\|_s \geq c\int_R^\infty r^{n-s(n-2)}\frac{dr}{r}=\infty.
$$
Similarly, we also deduce $v \not\in L^s(R^n)$ for all $s \leq \frac{n}{n-2}$.
(R1) is proved.

(2) For $r>0$, write
$$
w(x)=K(x)\int_{B_r(x)}\frac{w^p(y)dy}{|x-y|^{n-2}}
+K(x)\int_{R^n \setminus B_r(x)}\frac{w^p(y)dy}{|x-y|^{n-2}}
:=K_1+K_2.
$$
Clearly, for a suitably small $\epsilon>0$, from (R1) we deduce
$$
K_1 \leq C\|w\|_{1/\epsilon}^p (\int_0^d r^{n-\frac{n-2}{1-p\epsilon}}
\frac{dr}{r})^{1-p\epsilon} \leq C.
$$
On the other hand, by virtue of (\ref{base}) and (R1), we get $w \in L^p(R^n)$. Thus
$$
K_2 \leq Cr^{2-n}\|w\|_p^p \leq C.
$$
Combining the estimates of $K_1$ and $K_2$, we know that $w$ is
bounded. Thus, $u,v$ are bounded.

Next, we show that $w$ is decaying. Take $x_0 \in R^n$. By
exchanging the order of the integral variables, we have
$$
w(x_0) =(n-\alpha)K(x_0)\int_0^\infty (\frac{\int_{B_t(x_0)}
w^p(y)dy}{t^{n-\alpha}})\frac{dt}{t}.
$$
Since $w \in L^\infty(R^n)$, $\forall \varepsilon>0$, there exists
$\delta \in (0,1/2)$ such that
\begin{equation} \label{2.*3}
\int_0^{\delta} [\frac{\int_{B_t(x_0)}w^p(z)dz}{t^{n-\alpha}}]
\frac{dt}{t} \leq C\|w\|_\infty^p \int_0^{\delta} t^\alpha
\frac{dt}{t} <\varepsilon.
\end{equation}
As $|x-x_0|<\delta$,
$$
\int_{\delta}^{\infty} [\frac{\int_{B_t(x_0)}
w^p(z)dz}{t^{n-\alpha}}] \frac{dt}{t} \leq \int_{\delta}^{\infty}
[\frac{\int_{B_{t+\delta}(x)} w^p(z)dz}{(t+\delta)^{n-\alpha}}]
(\frac{t+\delta}{t})^{n-\alpha+1} \frac{d(t+\delta)}{t+\delta}\\[3mm]
\leq Cw(x).
$$
Combining this result with (\ref{2.*3}), we get
$$
w(x_0)<C\varepsilon+Cw(x), \quad for \quad |x-x_0| <\delta.
$$
Let $s=\frac{n(p-1)}{2}$, then
\begin{equation}
\begin{array}{ll}
w^{s}(x_0) &=|B_\delta(x_0)|^{-1}\displaystyle\int_{B_\delta(x_0)}
w^{s}(x_0)dx\\[3mm] &\leq
C\varepsilon^{s}
+C|B_\delta(x_0)|^{-1}\displaystyle\|w\|_{L^s(B_\delta(x_0))}^s.
\end{array}
\label{4.15}
\end{equation}
Since $w \in L^s(R^n)$, $\lim_{|x_0| \to
\infty}\|w\|_{L^s(B_\delta(x_0))}=0$. Inserting this result into
(\ref{4.15}), we have
$$
\lim_{|x_0| \to \infty}w^{s}(x_0)=0.
$$
This result means $u,v$ converge to zero when $|x| \to \infty$.
(R2) is proved.

(3) By the same proof of Theorem \ref{prop2.1} with $h=0$, we can
see the conclusion (R3).
\end{proof}

\begin{theorem} \label{th3.2}
If $u \in L^{\frac{n(p-1)}{2}}(R^n)$ solves (\ref{Sys}) in the classical sense,
then $u^pv \in L^1(R^n)$ and $u \in \mathcal{D}^{1,2}(R^n)$. Moreover,
\begin{equation} \label{yulian}
\int_{R^n}|\nabla u|^2dx=\int_{R^n}|\nabla v|^2dx=\sqrt{p}\int_{R^n}u^pvdx.
\end{equation}
\end{theorem}

\begin{proof}
{\it Step 1.}
By the H\"older inequality, from Theorem \ref{th3.1} (R1) and (\ref{base}), we can
deduce that $u^pv \in L^1(R^n)$.

{\it Step 2.} Take smooth function $\zeta(x)$ satisfying
$$
\left \{
   \begin{array}{lll}
&\zeta(x)=1, \quad &for~ |x| \leq 1;\\
&\zeta(x) \in [0,1], \quad &for~ |x| \in [1,2];\\
&\zeta(x)=0, \quad &for~ |x| \geq 2.
   \end{array}
   \right.
$$
Define the cut-off function
\begin{equation} \label{cut}
\zeta_R(x)=\zeta(\frac{x}{R}).
\end{equation}

Multiplying the first equation of (\ref{Sys}) by $u\zeta_R^{2}$ and integrating
on $D:=B_{3R}(0)$, we have
$$
-\int_D\zeta_R^{2}u\Delta udx=\sqrt{p}\int_D u^pv\zeta_R^{2}dx.
$$
Integrating by parts, we obtain
\begin{equation} \label{LY1}
\int_D|\nabla u|^2\zeta_R^{2}dx+2\int_D u\zeta_R \nabla u
\nabla\zeta_R dx=\sqrt{p}\int_D u^pv\zeta_R^{2}dx.
\end{equation}
Applying the Cauchy inequality, we get
\begin{equation} \label{LY2}
|\int_D u\zeta_R \nabla u \nabla\zeta_R dx| \leq \delta
\int_D|\nabla u|^2 \zeta_R^{2}dx +C\int_D u^2 |\nabla\zeta_R|^2 dx
\end{equation}
for any $\delta \in (0,1/2)$. According to Theorem \ref{th3.1} (R1),
it follows $u,v \in L^{2^*}(R^n)$.
By the H\"older inequality, we obtain
\begin{equation} \label{LY3}
\int_D u^2|\nabla\zeta_R|^2 dx \leq (\int_D u^{2^*}dx)^{1-2/n}
(\int_D|\nabla\zeta_R|^n dx)^{2/n} \leq C.
\end{equation}
Noting $u^pv \in L^1(R^n)$, from (\ref{LY1})-(\ref{LY3}) we deduce
$$
\int_D|\nabla u|^2\zeta_R^{2}dx \leq C.
$$
Letting $R \to \infty$ yields
$$
\int_{R^n}|\nabla u|^2dx <\infty.
$$
Similarly, we also obtain
$$
\int_{R^n}|\nabla v|^2dx <\infty.
$$
Combining the results above, we can see
$$
\int_{R^n}(u^{2^*}+v^{2^*}+u^pv+|\nabla u|^2+|\nabla v|^2)dx<\infty.
$$
Therefore, we can find $R_j$ such that
\begin{equation} \label{conv}
\lim_{R_j \to \infty}R_j\int_{\partial B_{R_j}}(u^{2^*}+v^{2^*}+u^pv
+|\nabla u|^2+|\nabla v|^2)ds =0.
\end{equation}

{\it Step 3.} Multiplying the first equation of (\ref{Sys}) by $u$
and integrating on $D$, we get
\begin{equation} \label{renj}
\int_D|\nabla u|^2dx-\int_{\partial D}u\partial_{\nu}uds
=\sqrt{p}\int_D u^pvdx.
\end{equation}
Here $\nu$ is the unit outward norm vector on $\partial D$.
By the H\"older inequality, from (\ref{conv}) we deduce
$$
|\int_{\partial D}u\partial_{\nu}uds| \leq CR^{\frac{n-1}{n}-1/2-1/2^*}(R\int_{\partial D}
|\nabla u|^2ds)^{1/2} (R\int_{\partial D}u^{2^*}ds)^{1/2^*} \to 0
$$
when $R=R_j \to \infty$. Letting $R=R_j \to \infty$ in (\ref{renj}), we have
$\|\nabla u\|_2^2=\sqrt{p}\|u^pv\|_1$. Similarly, we can also obtain
$\|\nabla v\|_2^2=\sqrt{p}\|u^pv\|_1$.
\end{proof}

The following result shows
that there does not exist $L^{\frac{n(p-1)}{2}}(R^n)$-solution if
$p$ is not equal to the critical exponent $2^*-1$.

\begin{theorem} \label{th3.3}
If a classical solution $u$ belongs to $L^{\frac{n(p-1)}{2}}(R^n)$, then
$p=2^*-1$, and hence $L^{n(p-1)/2}(R^n)=L^{2^*}(R^n)$.
\end{theorem}

\begin{proof}
Write $B=B_R(0)$.
Multiply (\ref{Sys}) by $x \cdot \nabla u$ and $x \cdot \nabla v$, respectively.
Integrating on $B$, we get
$$
-\int_B(x\cdot \nabla u)\Delta udx=\frac{1}{\sqrt{p}}\int_B
v(x\cdot \nabla u^p)dx,
$$
$$
-\int_B(x\cdot \nabla v)\Delta vdx=\sqrt{p}\int_Bu^p(x\cdot \nabla v)dx.
$$
Integrating by parts yields
$$\begin{array}{ll}
&-p\displaystyle\int_{\partial B}|x||\partial_{\nu}u|^2ds
+\frac{p}{2}\int_{\partial B}|x||\nabla u|^2ds
-\frac{n-2}{2}p\int_B|\nabla u|^2dx\\[3mm]
&=\sqrt{p}\displaystyle\int_B v(x\cdot \nabla u^p)dx,
\end{array}
$$
and
$$\begin{array}{ll}
&-\displaystyle\int_{\partial B}|x||\partial_{\nu}v|^2ds
+\frac{1}{2}\int_{\partial B}|x||\nabla v|^2ds
-\frac{n-2}{2}\int_B|\nabla v|^2dx\\[3mm]
&=\sqrt{p}\displaystyle\int_Bu^p(x\cdot \nabla v)dx.
\end{array}
$$
Adding two results together and integrating by parts again, we obtain
$$\begin{array}{ll}
&-\displaystyle\int_{\partial B}|x|(p|\partial_{\nu}u|^2+|\partial_{\nu}v|^2)ds
+\frac{1}{2}\int_{\partial B}|x|(p|\nabla u|^2+|\nabla v|^2)ds\\[3mm]
&-\displaystyle\frac{n-2}{2}\int_B(p|\nabla u|^2+|\nabla v|^2)dx\\[3mm]
&=\sqrt{p}\displaystyle\int_B x\cdot \nabla (u^pv)dx=\sqrt{p}\int_{\partial B}
|x|u^pvds-n\sqrt{p}\int_Bu^pvdx.
\end{array}
$$
Letting $R=R_j \to \infty$ and using (\ref{conv}), we have
$$
\frac{n-2}{2}\int_B(p|\nabla u|^2+|\nabla v|^2)dx
=n\sqrt{p}\int_Bu^pvdx.
$$
By (\ref{yulian}) we see $p=2^*-1$ finally.
\end{proof}

\begin{theorem} \label{th3.4}
If a classical solution $u$ of (\ref{Sys})
belongs to $L^{\frac{n(p-1)}{2}}(R^n)$, then
\begin{equation} \label{cla}
u(x)=v(x)=c(\frac{t}{t^2+|x-x^*|^2})^{\frac{n-2}{2}}.
\end{equation}
with some constant $c=c(n)$ and for some $t>0$.
\end{theorem}

\begin{proof}
{\it Step 1.} We claim $u \equiv v$.

Let $W=u-v$. By Theorems \ref{th3.1}
and \ref{th3.2}, we see
$$
\int_{R^n}(|W|^{2^*}+ |\nabla W|^2)dx<\infty.
$$
Thus, when $R=R_j \to \infty$,
$$
R\int_{\partial B_R(0)} (|W|^{2^*}+ |\partial_{\nu} W|^2) ds \to 0.
$$
From (\ref{Sys}), it follows
$\Delta W=\sqrt{p}u^{p-1}W$. Therefore,
\begin{equation} \label{jidu}
\int_B |\nabla W|^2dx+\sqrt{p}\int_B u^{p-1}W^2dx
=\int_{\partial B}W \partial_{\nu}W ds.
\end{equation}
Here $B=B_R(0)$.
By the H\"older inequality, as $R=R_j \to \infty$,
$$\begin{array}{ll}
&\quad |\displaystyle\int_{\partial B}W \partial_{\nu}W ds|\\[3mm]
&\leq C(R\displaystyle\int_{\partial B}|W|^{2^*}ds)^{1/2^*}
(R\int_{\partial B} |\partial_{\nu}W|^2 ds)^{1/2}
R^{(n-1)(1/2-1/2^*)-1/2-1/2^*}\\[3mm]
& \to 0.
\end{array}
$$
Inserting this into (\ref{jidu}) with $R=R_j \to \infty$,
we get
$$
\int_{R^n}(|\nabla W|^2+\sqrt{p}u^{p-1}W^2)dx=0,
$$
which implies $u \equiv v$.

{\it Step 2.} By virtue of $u \equiv v$ and Theorem \ref{th3.3},
(\ref{Sys}) is reduced to the single equation
$$
-\Delta u=\sqrt{p} u^{2^*-1}, \quad u>0 ~in ~R^n.
$$
According to the classification results in \cite{CL}, the positive
solutions must be the form as (\ref{cla})
in the critical case.
\end{proof}

The argument above implies $u \in L^{\frac{n(p-1)}{2}}(R^n)$ is
equivalent to (\ref{cla}).

At last, we complete the proof of Theorem \ref{th1.2}.

\begin{theorem} \label{th3.5}
If the classical solution $u$ of (\ref{Sys}) belongs to
$L^{\frac{n(p-1)}{2}}(R^n)$, if and only if one of the following
items holds

(C1) $u \in L^{2^*}(R^n)$ and $p=2^*-1$;

(C2) $u \in L^\infty(R^n)$ and $u(x)=O(|x|^{2-n})$ as $|x| \to
\infty$;

(C3) $u \in \mathcal{D}^{1,2}(R^n)$.
\end{theorem}

\begin{proof}
The necessity can be seen in Theorems \ref{th3.1}-\ref{th3.4}.
Next, we state the sufficiency.

Both (C1) and (C2) can lead easily to $u \in
L^{\frac{n(p-1)}{2}}(R^n)$. So, $u \in L^{\frac{n(p-1)}{2}}(R^n)$
is equivalent to (C1) and (C2).

Finally, (C3) implies (C1). In fact, by the Sobolev inequality we get
$u \in L^{2^*}(R^n)$. On the other hand, using
(\ref{LY1})-(\ref{LY3}) we can deduce $u^pv \in L^1(R^n)$ from $u
\in \mathcal{D}^{1,2}(R^n)$. Thus, (\ref{conv}) holds, and the
proof of Theorem \ref{th3.3} still works. This shows $p=2^*-1$.
\end{proof}

\paragraph{Remark.}
$u(x)=O(|x|^{2-n})$ as $|x| \to \infty$ in condition (C2) implies
$0<c_1 \leq u(x)|x|^{n-2} \leq c_2$ as $|x| \to \infty$. It is led
to by (\ref{aq}).

\section{Supercritical case}

If (\ref{base}) holds, then (\ref{Sys}) has a singular solution.

\begin{theorem} \label{th4.1}
Let $u(x)=\frac{c_1}{|x|^{t_1}}$ and $v(x)=\frac{c_2}{|x|^{t_2}}$
solve (\ref{Sys}) on $R^n \setminus \{0\}$, where $c_1,c_2,t_1,t_2$ are positive constants.
Then such singular solutions must be
$$
u(x)=v(x)=\frac{c}{|x|^{\frac{2}{p-1}}}
$$
with $c=[\frac{2n}{\sqrt{p}(p-1)}-\frac{4\sqrt{p}}{(p-1)^2}]^{\frac{1}{p-1}}$.
\end{theorem}

\begin{proof}
Write $U(r)=U(|x|)=u(x)$ and $V(r)=V(|x|)=v(x)$. Thus,
$$\begin{array}{ll}
&-U''-\frac{n-1}{r}U'=\frac{c_1t_1}{r^{t_1+2}}(n-t_1-2),
\\[3mm]
&-V''-\frac{n-1}{r}V'=\frac{c_2t_2}{r^{t_2+2}}(n-t_2-2).
\end{array}
$$
Since $u,v$ solve (\ref{Sys}), it is easy to get
$$
\left \{
   \begin{array}{l}
      t_1+2=t_1(p-1)+t_2,\\
      t_2+2=t_1p,
   \end{array}
   \right.
$$
and
$$
\left \{
   \begin{array}{l}
      c_1t_1(n-t_1-2)=\sqrt{p}c_1^{p-1}c_2,\\
      c_2t_2(n-t_2-2)=\sqrt{p}c_1^p
   \end{array}
   \right.
$$
as long as $n-2>\max\{t_1,t_2\}$.
Therefore, $t_1=t_2=\frac{2}{p-1}$ and $c_1=c_2=
[\frac{2n}{\sqrt{p}(p-1)}-\frac{4\sqrt{p}}{(p-1)^2}]^{\frac{1}{p-1}}$.
In view of (\ref{base}), the condition $n-2>\max\{t_1,t_2\}$
is true obviously.
\end{proof}

\paragraph{Remark.}
The fact $u \equiv v$ is natural by an analogous argument in Step 4 of the proof
of Theorem \ref{th4.4}.

\vskip 5mm

Next, we search for bounded solutions.

An example is the pair of cylinder-shaped solutions $(u^*,v^*)$ (cf. \cite{CLO4}).
Let
\begin{equation} \label{wlty}
u_*(x)=c(\frac{t}{t^2+|x-x^*|^2})^{\frac{n-2}{2}}.
\end{equation}
According to Theorem
\ref{th1.2}, $u_*$ solves (\ref{Sys}) in the whole space $R^n$
in the critical case $p=\frac{n+2}{n-2}$. Thus, it is not difficult to
see that $u^*(x,x_{n+1})=u_*(x)$ and $v^*(x_0,x)=u_*(x)$
still solves (\ref{Sys}) in $R^{n+1}$. In view of
$p>\frac{n+3}{n-1}$, the problem (\ref{Sys}) which $u^*$ satisfies in $R^{n+1}$
is equipped with the supercritical exponent.
Clearly, this pair of positive solution $(u^*,v^*)$ is neither radial nor decaying when $|x| \to \infty$.
We also see $u^* \neq v^*$ because the generating lines of the cylinders are different.

\vskip 5mm
At last, we search for a radial bounded solution decaying with the slow rate
$\frac{2}{p-1}$.

In order to find the existence of entire solutions in $R^n$, we need
the following nonexistence result on a bounded domain.
It can be verified by the Pohozaev identity.

\begin{theorem} \label{th6.3}
Let $D \subset R^n$ be a ball centered at the origin. If
\begin{equation} \label{nc}
p \geq 2^*-1,
\end{equation}
then the following boundary value problem has no
positive radial solution in $C^{2}(D)\cap C^{1}(\bar{D})$
\begin{equation} \label{L1}
    -\Delta u=\sqrt{p}u^{p-1}v, \quad in \quad D,
\end{equation}
\begin{equation} \label{L2}
    -\Delta v=\sqrt{p}u^p, \quad in \quad D,
\end{equation}
\begin{equation} \label{L3}
    u=v=0 \quad on \quad \partial D.
\end{equation}
\end{theorem}

\begin{proof}
Suppose that $u,v$ are positive radial solutions, we will
deduce a contradiction.

Multiply (\ref{L1})
and (\ref{L2}) by $u$ and $v$, respectively. Integrating on $D$
and using (\ref{L3}), we have
\begin{equation} \label{L4}
\int_D|\nabla u|^2dx=\int_D|\nabla v|^2dx=\sqrt{p}\int_Du^pvdx.
\end{equation}

Since $u$ has the radial structure, $|\nabla u|^2=|\partial_{\nu} u|^2$ on
$\partial D$. Here $\nu$ is the unit outward norm vector on $\partial D$.

Multiplying (\ref{L1}) by $(x\cdot \nabla u)$ and integrating on $D$,
we get
$$
-\int_{\partial D}|x||\partial_{\nu}u|^2ds+\int_D|\nabla u|^2dx
+\frac{1}{2}\int_Dx\cdot \nabla(|\nabla u|^2)dx
=\frac{1}{\sqrt{p}}\int_Dv(x\cdot \nabla u^p)dx.
$$
Integrating by parts and noting (\ref{L3}), we obtain
$$
\frac{1}{2}\int_{\partial D}|x||\partial_{\nu}u|^2ds
+\frac{n-2}{2}\int_D|\nabla u|^2dx
=\frac{n}{\sqrt{p}}\int_D u^pvdx
+\frac{1}{\sqrt{p}}\int_Du^p(x\cdot \nabla v)dx.
$$
Similarly, from (\ref{L2}) we also deduce that
$$
-\frac{1}{2}\int_{\partial D}|x||\partial_{\nu}v|^2ds
-\frac{n-2}{2}\int_D|\nabla v|^2dx
=\sqrt{p}\int_Du^p(x\cdot \nabla v)dx.
$$
Combining two results above with (\ref{L4}) yields
$$
-\frac{1}{2}\int_{\partial D}|x|(|\partial_{\nu}v|^2
+\frac{1}{p}|\partial_{\nu}v|^2)ds
=\frac{n-2}{2}(\sqrt{p}+\frac{1}{\sqrt{p}})\int_Du^pvdx
-\frac{n}{\sqrt{p}}\int_D u^pvdx.
$$
Therefore,
$$
\frac{n-2}{2}(\sqrt{p}+\frac{1}{\sqrt{p}})-\frac{n}{\sqrt{p}}<0,
$$
which contradicts with (\ref{nc}).
\end{proof}

Based on the Liouville type result above, we can
search for positive solutions of (\ref{Sys}) with radial structures.
Let $u,v$ be radially symmetric about $x^* \in R^n$. We can write
$$\begin{array}{ll}
&U(r)=U(|x-x^*|)=u(x-x^*), \\[3mm]
&V(r)=V(|x-x^*|)=v(x-x^*).
\end{array}
$$

\begin{theorem} \label{th6.1}
Let $p \geq \max\{2,2^*-1\}$. Then the following ODE system
\begin{equation}
 \left \{
   \begin{array}{l}
      -(U''+\frac{n-1}{r}U')=\sqrt{p}U^{p-1}V, \quad -(V''+\frac{n-1}{r}V')=\sqrt{p}U^p,
      \quad r>0\\
      U'(0)=V'(0)=0, \quad U(0)=1, \quad V(0)=a,
   \end{array}
   \right. \label{ODE}          
 \end{equation}
has entire solutions for some positive constant $a$.
\end{theorem}

\begin{proof}
Here we use the shooting method.

{\it Step 1.} By the standard contraction argument, by $p \geq 2$ we can see the
local existence. We denote the solutions by $u_a(r),v_a(r)$.

{\it Step 2.} We claim that either (\ref{ODE}) has entire solutions for all $a>1$, or
for some $a^*>1$, there exists $R \in (0,1]$ such that
$u_{a^*}(r),v_{a^*}(r)>0$ for $r \in [0,R)$ and $u_{a^*}(R)=0$.

In fact, integrating (\ref{ODE}) twice yields
$$
v_a(r)=v_a(0)-\sqrt{p}\int_0^r\tau^{1-n}\int_0^\tau s^{n-1}u_a^p(s)dsd\tau,
$$
and
$$
u_a(r)=u_a(0)-\sqrt{p}\int_0^r\tau^{1-n}\int_0^\tau s^{n-1}u_a^{p-1}(s)v_a(s)dsd\tau.
$$
Thus,
\begin{equation} \label{minus}
v_a(r)-u_a(r)=(a-1)+\sqrt{p}\int_0^r\tau^{1-n}\int_0^\tau s^{n-1}u_a^{p-1}(s)(v_a(s)-u(s))dsd\tau.
\end{equation}
Let $a>1$. So we can find $\delta>0$ such that
$v_a(r)>u_a(r)$ for $r \in [0,\delta)$ by the continuity of $u_a,v_a$.

We claim $v_a(r)>u_a(r)$ for all $r \geq 0$.
Otherwise, there exists $r_0 \geq \delta$ such that $v(r_0)=u(r_0)$ and $v(r)>u(r)$ as $r \in [0,r_0)$.
From (\ref{minus}) with $r=r_0$ we can deduce a contradiction easily.

Therefore, if $u_a(r)>0$ for all $r \geq 0$, then (\ref{ODE}) has entire solutions and the proof is complete.
Otherwise, we can find $R>0$ such that $u_a(r),v_a(r)>0$ for $r \in (0,R)$ and $u_a(R)=0$.
We denote the $a$  in this state by $a^*$.

{\it Step 3.} We claim that for $a<\varepsilon_0=\frac{1}{n2^{1+p}\sqrt{p}}$, there
exists $R \in (0,1]$, such that $u_a(r),v_a(r)>0$ for $r \in [0,R)$ and
$v_a(R)=0$.

In fact, from (\ref{ODE}) we obtain $u(r)>v(r)$ for $r \geq 0$ and $u_a',v_a'<0$
for $r>0$. Thus,
$v_a(r) \leq a<\varepsilon_0$ and $u_a(r) \leq u_a(0)=1$. Thus,
$$
u_a(r)=u_a(0)-\sqrt{p}\int_0^r\tau^{1-n}\int_0^\tau s^{n-1}u_a^{p-1}(s)v_a(s)dsd\tau
\geq 1-\frac{\varepsilon_0\sqrt{p}r^2}{2n} \geq \frac{1}{2},
$$
for $r \in (0,1)$. Therefore,
$$
v_a(r)=v_a(0)-\sqrt{p}\int_0^r\tau^{1-n}\int_0^\tau s^{n-1}u_a^p(s)dsd\tau
< \varepsilon_0-\frac{\sqrt{p}}{2^p}\frac{r^2}{2n}.
$$
This proves that for $a < \varepsilon_0$, we can find $R \in
(0,1]$ such that $u_a(r),v_a(r)>0$ for $r \in (0,R)$ and
$v_a(R)=0$.

{\it Step 4.} Let $\underline{a}=\sup \underline{S}$, where
$$\begin{array}{ll}
\underline{S}:=\{\varepsilon; &\hbox{when}~ a \in (0,\varepsilon),
\quad\exists R_a>0, ~\hbox{such that}\\
&u_a(r)>0, v_a(r) \geq 0, ~\hbox{for}~ r \in [0,R_a], v_a(R_a)=0\}.
\end{array}
$$
Clearly, $\underline{S} \neq \emptyset$ by virtue of $\varepsilon_0
\in \underline{S}$. From Step 2, it follows $\varepsilon \leq a^*$ for $\varepsilon
\in \underline{S}$. Namely, $\underline{S}$ is upper bounded, and hence
we see the existence of $\underline{a}$.

{\it Step 5.} Write $\bar{u}(r)=u_{\underline{a}}(r)$ and
$\bar{v}(r)=v_{\underline{a}}(r)$. We claim that $\bar{u}(r),\bar{v}(r)>0$
for $r \in [0,\infty)$, and hence they are entire positive solutions of
(\ref{ODE}).

Otherwise, there exists $\bar{R}>0$ such that $\bar{u}(r),\bar{v}(r)>0$
for $r \in (0,\bar{R})$ and one of the following consequences holds:
$$\begin{array}{ll}
&(1)~~ \bar{u}(\bar{R})=0, \bar{v}(\bar{R})>0;\\[3mm]
&(2)~~ \bar{v}(\bar{R})=0, \bar{u}(\bar{R})>0;\\[3mm]
&(3)~~ \bar{u}(\bar{R})=0, \bar{v}(\bar{R})=0.
\end{array}
$$
We deduce the contradictions from three consequences above.

(1) By $C^1$-continuous dependence of $u_a,v_a$ in $a$, and the fact
$\bar{u}'(\bar{R})<0$, we see that for all $|a-\underline{a}|$ small,
there exists $R_a>0$ such that
$$\begin{array}{ll}
&\bar{u}(r),\bar{v}(r)>0, \quad for ~r \in (0,R_a);\\
&\bar{u}(R_a)=0, \quad \bar{v}(R_a)>0.
\end{array}
$$
This contradicts with the definition of $\underline{a}$.

(2) Similarly, for $|a-\underline{a}|$ small, there exists $R_a>0$ such that
$$\begin{array}{ll}
&\bar{u}(r),\bar{v}(r)>0, \quad for ~r \in (0,R_a);\\
&\bar{u}(R_a)>0, \quad \bar{v}(R_a)=0.
\end{array}
$$
This implies that $\underline{a}+\delta \in \underline{S}$ for some $\delta>0$,
which contradicts with the definition of $\underline{a}$.

(3) The consequence implies that $u(x)=\bar{u}(|x|)$ and $v(x)=\bar{v}(|x|)$
are solutions of the system
\begin{equation} \label{bs}
 \left \{
   \begin{array}{l}
      -\Delta u=\sqrt{p}u^{p-1}v, \quad -\Delta v=\sqrt{p}u^p, ~in ~B_R,\\
      u,v>0 ~in ~B_R, \quad u=v=0 ~on ~\partial B_R.
   \end{array}
   \right.
\end{equation}
It is impossible by Theorem \ref{th6.3}.

All the contradictions show that our claim is true. Thus, the entire positive solutions
exist.
\end{proof}

\begin{theorem} \label{th4.4}
The entire solutions obtained in Theorem \ref{th6.1} satisfy
\begin{equation} \label{jiaolv}
\lim_{|x| \to \infty}u(x)=\lim_{|x| \to \infty}v(x)=0
\end{equation}
and $u \equiv v$. Moreover, either
$$
p=2^*-1\quad \textrm{and} \quad u=u_*,
$$
or
$$
p>2^*-1\quad \textrm{and} \quad c_1 \leq
u(x)|x|^{\frac{2}{p-1}} \leq c_2
$$
when $|x| \to \infty$, where
$c_1,c_2$ are positive constants.
\end{theorem}

\begin{proof}
{\it Step 1.}
We claim $\lim_{r \to \infty}\bar{u}(r)=0$.

Eq. (\ref{ODE}) implies $\bar{u}'<0$ and $\bar{v}'<0$ for $r>0$. So $\bar{u}$
and $\bar{v}$ are decreasing positive solutions, and
$\lim_{r \to \infty}\bar{u}(r)$, $\lim_{r \to \infty}\bar{v}(r)$ exist.

If there exists $c>0$ such that $\bar{u}(r) \geq c$ for $r>0$, then (\ref{ODE})
shows that $\bar{v}$ satisfies
$$
\bar{v}''+\frac{n-1}{r}\bar{v}' \leq -\sqrt{p}c^p.
$$
Integrating twice yields
$$
\bar{v}(r) \leq \bar{v}(0)-\frac{\sqrt{p}c^pr^2}{2n}
$$
for $r>0$. It is impossible since $\bar{v}$ is an entire positive solution.
This shows that $\bar{u} \to 0$ when $r \to \infty$.

{\it Step 2.}
Furthermore, we claim $u(x)\leq c_2|x|^{\frac{-2}{p-1}}$
when $|x| \to \infty$.

In fact, since $u(x),v(x)$ are the positive entire solutions of
(\ref{Sys}), they satisfy (\ref{IE}) according to the Remark
in section 1. Moreover,
$u(x),v(x)$ are radially symmetric and decreasing,
we can deduce
$$
v(x) \geq cu^p(x)\int_0^{|x|}r^2\frac{dr}{r} \geq c|x|^2u^p(x),
$$
and hence
$$
u(x) \geq cu^{p-1}(x)v(x)\int_0^{|x|}r^2\frac{dr}{r}
\geq c|x|^2u^{p-1}(x)v(x) \geq c|x|^4u^{2p-1}(x).
$$
This implies our claim.

{\it Step 3.} We claim $\lim_{|x| \to \infty}v(x)=0$.

Clearly, for large $R>0$, we obtain from (\ref{IE}) that
$$\begin{array}{ll}
v(x)&=c_2(\displaystyle\int_{B_R(0)}\frac{u^p(y)dy}{|x-y|^{n-2}}
+\int_{B_{|x|/2}(x)}\frac{u^p(y)dy}{|x-y|^{n-2}}\\[3mm]
&\quad +\displaystyle\int_{B_R^c(0)\setminus B_{|x|/2}(x)}
\frac{u^p(y)dy}{|x-y|^{n-2}})\\[4mm]
&:=c_2(I_1+I_2+I_3).
\end{array}
$$

First, $u$ is bounded. So, as $|x| \to \infty$,
$$
I_1 \leq C\int_{B_R(0)}\frac{dy}{|x-y|^{n-2}} \leq C|x|^{2-n} \to 0.
$$
Next, for large $|x|$, Step 2 implies
$u(y) \leq c|x|^{\frac{-2}{p-1}}$ as $y \in B_{|x|/2}(x)$. Therefore,
$$
I_2 \leq c|x|^{\frac{-2p}{p-1}}\int_{B_{|x|/2}(x)}\frac{dy}{|x-y|^{n-2}}
\leq C|x|^{2-\frac{2p}{p-1}} \to 0.
$$
Finally, using the Young inequality and letting $R \to \infty$, we have
$$
I_3 \leq C\int_{B_R^c(0)\setminus B_{|x|/2}(x)}\frac{dy}{|x-y|^{n-2}|y|^{\frac{2p}{p-1}}}
\leq C\int_R^\infty r^{2-\frac{2p}{p-1}}\frac{dr}{r} \to 0.
$$

Combining the estimates of $I_1,I_2$ and $I_3$, we get
$$
\lim_{|x| \to \infty}v(x)=0.
$$
This result, together with Step 1, implies (\ref{jiaolv}).

{\it Step 4.} We claim $u \equiv v$. The argument in Step 1 of the proof
of Theorem \ref{th3.4} does not work, since the boundary integral
is difficult to handle. We use the idea in \cite{LM} to prove this uniqueness.
This idea is universally valid, which implies the uniqueness as long as
(\ref{jiaolv}) holds. In addition, this idea is also independent whether the
value of $p$ is critical or not.

For any $r_0 \geq 0$, we prove $U(r_0)=V(r_0)$. Otherwise, either
\begin{equation} \label{fou1}
U(r_0)<V(r_0),
\end{equation}
or
\begin{equation} \label{fou2}
U(r_0)>V(r_0).
\end{equation}

If (\ref{fou1}) holds, by the continuity of $U$ and $V$, we can find $R>0$ such that
\begin{equation} \label{LL1}
U(r)<V(r) \quad as \quad r \in [r_0,R).
\end{equation}
Set $R_0=\sup\{R; (\ref{LL1})$ is true$\}$. Thus, when $R_0<\infty$,
\begin{equation} \label{nojiya}
U(R_0)=V(R_0).
\end{equation}
In view of (\ref{jiaolv}), this conclusion still holds
even if $R_0=\infty$.

By (\ref{Sys}), we see that for $r>0$,
$$
-(r^{n-1}U')'=\sqrt{p}r^{n-1}U^{p-1}V, \quad -(r^{n-1}V')'=\sqrt{p}r^{n-1}U^p.
$$
Integrating twice, we get
$$
U(R_0)=U(r_0)-\sqrt{p}\int_{r_0}^{R_0}r^{1-n}\int_0^r s^{n-1}U^{p-1}(s)V(s)dsdr,
$$
$$
V(R_0)=V(r_0)-\sqrt{p}\int_{r_0}^{R_0}r^{1-n}\int_0^r s^{n-1}U^p(s)dsdr.
$$
Using (\ref{fou1}), (\ref{LL1}) and (\ref{nojiya}), we obtain
$$
0>U(r_0)-V(r_0)=\sqrt{p}\int_{r_0}^{R_0}r^{1-n}\int_0^r s^{n-1}U^{p-1}(s)(V(s)-U(s))dsdr \geq 0,
$$
which is impossible. So (\ref{fou1}) is not true. Similarly,
we also see that (\ref{fou2}) is not true. Since $r_0$ is arbitrary,
we know $u \equiv v$.

{\it Step 5.} By virtue of $u \equiv v$, (\ref{Sys}) is reduced to
the single equation
$$
-\Delta u=\sqrt{p} u^p, \quad u>0 ~in~R^n
$$
with (\ref{nc}). According to the results in \cite{YiLi} and
\cite{LN}, we know that $u$ either
decays fast
$$
c_1 \leq u(x)|x|^{n-2} \leq c_2,
$$
or decays slowly
$$
c_1 \leq u(x)|x|^{\frac{2}{p-1}} \leq c_2,
$$
when $|x| \to \infty$. Here $c_1,c_2$ are positive constants.

If $u$ decays fast, by Theorem \ref{th1.2} we know $p=2^*-1$ and
$u \equiv u_*$. Here $u_*$ is the radial function in (\ref{wlty}).

If $u$ decays slowly, we claim $p>2^*-1$. Otherwise, from (\ref{nc})
we have $p=2^*-1$. According to the classification result in \cite{CL},
$u \equiv u_*$. This contradicts the slow decay rate.
\end{proof}




\begin{thebibliography}{99}

\bibitem{Amst}
    A. Ambrosetti,
    \emph{On Schr\"odinger-Poisson systems}, Milan J.
    Math., \textbf{76} (2008),  257--274.

\bibitem{BAM}
    N. Ben Abdallah, F. Mehats,
    \emph{On a Vlasov-Schr\"odinger-Poisson model}, Comm. Partial Differential Equations,
    \textbf{29} (2004), 173--206.

\bibitem{CGS}
    L. Caffarelli, B. Gidas, J. Spruck,
    \emph{Asymptotic symmetry and local behavior of semilinear
    elliptic equations with critical Sobolev growth}, Comm. Pure
    Appl. Math. \textbf{42} (1989), 271--297.

\bibitem{Caz} T. Cazenave,
     \emph{Semilinear schr\"odinger equations},
     Courant Lecture Notes in Mathematics, 10. New York University,
     Courant Institute of Mathematical Sciences, New York;
     American Mathematical Society, Providence, RI, 2003.

\bibitem{CL}
    W. Chen, C. Li,
    \emph{Classification of solutions of some
    nonlinear elliptic equations}, Duke Math. J., \textbf{63} (1991),
    615--622.

\bibitem{CLO}
    W. Chen, C. Li, B. Ou,
    \emph{Classification
    of solutions for an integral equation}, Comm. Pure Appl.
    Math., \textbf{59} (2006),  330--343.

\bibitem{CLO4}
    W. Chen, C. Li, B. Ou,
    \emph{Qualitative properties
    of solutions for an integral equation}, Discrete Contin. Dyn. Syst.,
    \textbf{12} (2005), 347--354.

\bibitem{GNN}
    B. Gidas, W.-M. Ni, L. Nirenberg,
    \emph{Symmetry of positive solutions
    of nonlinear elliptic equations in $R^{n}$} (collected in the
    book \emph{Mathematical Analysis and Applications}, which is vol. 7a
    of the book series \emph{Advances in Mathematics. Supplementary
    Studies}, Academic Press, New York, 1981.)

\bibitem{GS}
      B. Gidas, J. Spruck,
      \emph{Global and local behavior of positive solutions of nonlinear elliptic
      equations}, Comm. Pure Appl. Math., \textbf{34}
    (1981), 525--598.

\bibitem{GV} J. Ginibre, G. Velo,
     \emph{Long range scattering and modified wave operators for some Hartree type equations},
     Rev. Math. Phys., \textbf{12} (2000), 361--429.

\bibitem{HW}
      E. Hebey, J. Wei,
      \emph{Schr\"odinger-Poisson systems in the 3-sphere}, Calc. Var. Partial
    Differential Equations, \textbf{47} (2013), 25--54.

\bibitem{JL-Z}
      L. Jeanjean, T. Luo,
      \emph{Sharp nonexistence results of prescribed $L^2$-norm solutions for some class
    of Schr\"odinger-Poisson and quasi-linear equations}, Z. Angew. Math. Phys.,
    \textbf{64} (2013), 937--954.

\bibitem{JL}
      C. Jin, C. Li,
      \emph{Qualitative analysis of
    some systems of integral equations}, Calc. Var. Partial
    Differential Equations, \textbf{26}
    (2006), 447--457.

\bibitem{Lei}
      Y. Lei,
      \emph{On the regularity of positive solutions of a class of
      Choquard type equations},
      Math. Z., \textbf{273} (2013), 883--905.

\bibitem{SIAM}
      Y. Lei,
      \emph{Qualitative analysis for the static Hartree-type equations},
      SIAM J. Math. Anal., \textbf{45} (2013), 388--406.

\bibitem{LLM}
      Y. Lei, C. Li, C. Ma,
      \emph{Asymptotic radial symmetry and growth estimates of
      positive solutions to weighted Hardy-Littlewood-Sobolev system},
      Calc. Var. Partial Differential Equations, \textbf{45}
      (2012), 43--61.

\bibitem{LM}  C. Li, L. Ma,
      \emph{Uniqueness of positive bound states to
      Schr\"odinger systems with critical exponents},
      SIAM J. Math. Anal., \textbf{40} (2008),  1049--1057.


\bibitem{LMZ} D. Li, C. Miao, X. Zhang,
      \emph{The focusing energy-critical Hartree equation},
      J. Differential Equations, \textbf{246} (2009), 1139--1163.

\bibitem{YLi}
    Y.-Y. Li,
    \emph{Remark on some conformally invariant integral equations:
    the method of moving spheres}, J. Eur. Math.
    Soc., \textbf{6} (2004), 153--180.

\bibitem{YiLi}
    Y. Li,
    \emph{Asymptotic behavior of positive solutions of equation $\Delta u+K(x)u^p=0$ in $R^n$},
    J. Differential Equations, \textbf{95} (1992), 304--330.

\bibitem{LN}
    Y. Li, W.-M. Ni,
    \emph{On conformal scalar curvature equations in $R^n$},
    Duke Math. J., \textbf{57} (1988), 895--924.

\bibitem{Lieb}
    E. Lieb,
    \emph{Sharp constants in the Hardy-Littlewood-Sobolev and
    related inequalities}, Ann. of Math., \textbf{118} (1983),
    349--374.

\bibitem{LS}
    E. Lieb, B. Simon,
    \emph{The Hartree-Fock theory for Coulomb systems},
    Comm. Math. Phys., \textbf{53} (1977), 185--194.

\bibitem{MZ} L. Ma, L. Zhao,
      \emph{Classification of positive solitary solutions
      of the nonlinear Choquard equation},
      Arch. Rational Mech. Anal., \text{195} (2010),
     455--467.

\bibitem{MVS}  V. Moroz, J. Van Schaftingen,
      \emph{Nonexistence and optimal decay of supersolutions
      to Choquard equations in exterior domains},
      J. Differential Equations, \textbf{254} (2013), 3089--3145.

\bibitem{Na} K. Nakanishi,
     \emph{Energy scattering for Hartree equations},
     Math. Res. Lett., \textbf{6} (1999), 107--118.

\bibitem{QS} P. Quittner, P. Souplet,
     \emph{Symmetry of components for semilinear elliptic systems},
     SIAM J. Math. Anal., \textbf{44} (2012), 2545--2559.

\bibitem{Smets}  D. Smets,
      \emph{Nonlinear Schr\"odinger equations with Hardy potential and critical nonlinearities},
      Trans. Amer. Math. Soc., \textbf{357} (2005), 2909--2938.


\end{thebibliography}
\end{document}